\documentclass[12pt]{amsart}    
\usepackage{amscd}

\def\cal#1{\mathcal{#1}}

\def\NZQ{\Bbb}               
\def\NN{{\NZQ N}}

\def\ZZ{{\NZQ Z}}
\def\PP{{\NZQ P}}

\def\PP{{\NZQ P}}

\def\frk{\frak}               

\def\mm{{\frk m}}

\def\opn#1#2{\def#1{\operatorname{#2}}} 
\opn\chara{char}
\opn\length{\ell}
\opn\pd{pd}
\opn\rk{rk}
\opn\projdim{proj\,dim}
\opn\rank{rank}
\opn\depth{depth}
\opn\grade{grade}
\opn\height{ht}
\opn\embdim{emb\,dim}
\opn\codim{codim}
\def\OO{\mathcal{O}}
\opn\Tr{Tr}
\opn\bigrank{big\,rank}
\opn\superheight{superheight}\opn\lcm{lcm}
\opn\trdeg{tr\,deg}%
\opn\reg{reg}
\opn\lreg{lreg}
\opn\div{div}
\opn\Div{Div}
\opn\WDiv{WDiv}
\opn\cl{cl}
\opn\Cl{Cl}
\opn\Spec{Spec}
\opn\Supp{Supp}
\opn\supp{supp}
\opn\Sing{Sing}
\opn\Ass{Ass}
\opn\Assh{Assh}
\opn\Min{Min}
\opn\Reg{Reg}
\opn\Ann{Ann}
\opn\Rad{Rad}
\opn\Soc{Soc}
\opn\Socle{Socle}
\opn\Ker{Ker}
\opn\Coker{Coker}
\opn\Im{Im}
\opn\Hom{Hom}
\opn\Mor{Mor}
\opn\Tor{Tor}
\opn\Ext{Ext}
\opn\End{End}
\opn\Aut{Aut}
\opn\id{id}

\opn\nat{nat}
\opn\pff{pf}
\opn\Pf{Pf}
\opn\GL{GL}
\opn\SL{SL}
\opn\mod{mod}
\opn\ord{ord}
\opn\Proj{Proj}
\opn\aff{aff}
\opn\con{conv}
\opn\relint{relint}
\opn\st{st}
\opn\lk{lk}
\opn\cn{cn}
\opn\core{core}
\opn\vol{vol}
\opn\link{link}
\opn\star{star}
\opn\gr{gr}

\def\pot#1#2{#1[\kern-0.28ex[#2]\kern-0.28ex]}

\opn\dirlim{\underrightarrow{\lim}}
\opn\inivlim{\underleftarrow{\lim}}
\let\dirsum=\oplus
\let\tensor=\otimes
\let\iso=\cong
\let\Union=\bigcup

\let\Dirsum=\bigoplus

\let\to=\rightarrow

\let\To=\longrightarrow
\let\oT=\longleftarrow
\def\Implies{\ifmmode\Longrightarrow \else
     \unskip${}\Longrightarrow{}$\ignorespaces\fi}
\def\implies{\ifmmode\Rightarrow \else
     \unskip${}\Rightarrow{}$\ignorespaces\fi}
\def\iff{\ifmmode\Longleftrightarrow \else
     \unskip${}\Longleftrightarrow{}$\ignorespaces\fi}

\let\:=\colon
\opn\H{H}
\opn\Pic{Pic}

\newtheorem{Theorem}{Theorem}
\newtheorem{Corollary}[Theorem]{Corollary}
\newtheorem{Lemma}[Theorem]{Lemma}
\newtheorem{Proposition}[Theorem]{Proposition}
\newtheorem{Remark}[Theorem]{Remark}

\newtheorem{Definition}[Theorem]{Definition}

\let\epsilon\varepsilon
\def\OO{{\cal O}} 

\opn\inii{in}
\opn\inim{inm}
\opn\set{set}
\def\pnt{{\raise0.5mm\hbox{\large\bf.}}}

\begin{document}

\title{Raynaud-Mukai construction and 
Calabi-Yau Threefolds in Positive Characteristic}
\author{Yukihide Takayama}
\address{Department of Mathematical
Sciences, Ritsumeikan University, 
1-1-1 Nojihigashi, Kusatsu, Shiga 525-8577, Japan}
\email{takayama@se.ritsumei.ac.jp}

\def\Coh#1#2{H_{\mm}^{#1}(#2)}
\def\eCoh#1#2#3{H_{#1}^{#2}(#3)}

\newcommand{\AppTh}{Theorem~\ref{approxtheorem} }
\def\da{\downarrow}
\newcommand{\ua}{\uparrow}
\newcommand{\namedto}[1]{\buildrel\mbox{$#1$}\over\rightarrow}
\newcommand{\bdel}{\bar\partial}
\newcommand{\proj}{{\rm proj.}}

\newenvironment{myremark}[1]{{\bf Note:\ } \dotfill\\ \it{#1}}{\\ \dotfill
{\bf Note end.}}
\newcommand{\transdeg}[2]{{\rm trans. deg}_{#1}(#2)}
\newcommand{\mSpec}[1]{{\rm m\hbox{-}Spec}(#1)}

\newcommand{\tbf}{{{\Large To Be Filled!!}}}

\pagestyle{plain}
\maketitle

\def\gCoh#1#2#3{H_{#1}^{#2}\left(#3\right)}
\def\subsetneq{\raisebox{.6ex}{{\small $\; \underset{\ne}{\subset}\; $}}}
\opn\Exc{Exc}

\def\HHom{{\cal Hom}}

\begin{abstract}
In this article, we study the possibility of producing a Calabi-Yau
threefold in positive characteristic which is a counter-example to
Kodaira vanishing. The only known method to construct the
counter-example is so called inductive method such as the Raynaud-Mukai
construction or Russel construction. We consider Mukai's method and
its modification.  Finally, as an application of Shepherd-Barron
vanishing theorem of Fano threefolds, we 
compute $H^1(X, H^{-1})$ for any ample line bundle $H$ on a Calabi-Yau threefold $X$
on which Kodaira vanishing fails.
\end{abstract}

\section{Introduction}

Although every K3 surface in positive characteristic can be lifted to 
characteristic $0$ \cite{D},  there are some non-liftable Calabi-Yau 
threefolds, namely a smooth threefold $X$ with trivial canonical bundle 
and $H^1(X, \OO_X) = H^2(X,\OO_X)=0$. 
If a Calabi-Yau polarized 
threefold $(X,L)$ over the field $k$ of $\chara(k)=p\geq 3$
is a counter-example to Kodaira vanishing,
i.e., $H^i(X, L^{-1})\ne{0}$ for $i=1$ or $i=2$, 
$X$ is non-liftable to the second Witt vector ring $W_2(K)$ (and the Witt vector ring $W(k)$)
by the cerebrated Raynaud-Deligne-Illusies version of Kodaira vanishing theorem \cite{DI}.
But this does not necessarily imply that $X$ cannot be liftable to characteristic~$0$.
Moreover, a non-liftable variety is  not necessarily a counter-example to
Kodaira vanishing and 
as far as the author is aware, it is not 
known whether Kodaira vanishing holds for the non-liftable Calabi-Yau 
threefolds \cite{Hi99,Hi07,Hi08,Sh03,Eke05,CS} that have been found so far.
We do not even know whether Kodaira vanishing holds for all Calabi-Yau
threefolds.
Thus Kodaira type vanishing for Calabi-Yau threefolds is an interesting 
problem, which is independent from but seems to be closely related to
the lifting problem.

A counter-example to Kodaira vanishing has been given by M.~Raynaud,
which is a surface over a curve \cite{Ray}. 
This example was extended to arbitrary dimension by S.~Mukai
\cite{Mu79, Mu05}, which we will call the Raynaud-Mukai construction or,
simply, Mukai construction.  

The idea is,  so to say, an inductive construction. Namely, we start from 
a polarized smooth curve $(C, D)$. The ample divisor $D$ satisfies a 
special condition, which is a sufficient condition for the non-vanishing
$H^1(X, \OO_X(-D))\ne{0}$, and called a (pre-)Tango
structure. Then we give an algorithm to construct from a 
variety $X$  with a (pre-)Tango structure $D$ a new  variety
$\tilde{X}$ with a higher dimensional (pre-)Tango structure $\tilde{D}$
such that $\dim\tilde{X} = \dim X + 1$, using cyclic cover technique.
There is another way of constructing counter-examples using quotient
of $p$-closed differential forms \cite{russ, TakeCol}).
But this is also an inductive construction and the obtained
varieties are the same as the Raynaud-Mukai construction \cite{TakeCol}.
As far as the author is aware, non-inductive
construction of higher dimensional counter-examples 
is not yet found.

In this paper, we consider the problem of whether 
we can construct a Calabi-Yau threefold
with Kodaira non-vanishing by Mukai construction or by its
modification. 
Section~2 presents the Raynaud-Mukai construction.
For $p\geq 5$, Raynaud-Mukai varieties are of general type
so that the only possibility resides in the cases of $p=2, 3$.
Then in section~3, we will see that Mukai construction does not produce any 
K3 surfaces or Calabi-Yau threefolds (Corollary~\ref{neverK3} and 
Corollary~\ref{trivialK}). 
Then we consider possible modifications
of the  Raynaud-Mukai construction: we keep the inductive construction but
give up obtaining a (pre-)Tango structure. We show that 
if there exists a surface $X$ of general type together with 
a (pre-)Tango structure $D$ satisfying some property 
(this is not obtained  by Mukai construction), we 
can construct a Calabi-Yau threefold $\tilde{X}$
with a (pre-)Tango structure $\tilde{D}$
(Corollary~\ref{requiredsurface}) and 
describe the cohomology $H^1(\tilde{X}, \OO_{\tilde{X}})$ 
in certain situations (Proposition~\ref{h1ofCY3fold}).
Unfortunately, we could not prove or disprove 
existence of such a polarized surface $(X, D)$.

Finally, in section~3 we show that if Kodaira non-vanishing 
$H^1(X, L^{-1})\ne{0}$ holds for a polarized Calabi-Yau threefold 
$(X,L)$ over the field $k$ of $\chara{k}=p\geq 5$
satisfying the condition that $L^\ell$ is a Tango-structure for some $\ell\geq 1$, 
we compute the cohomology $H^1(X, H^{-1})$ for any 
ample line bundle $H$ of $X$ (Theorem~\ref{main}, Corollary~\ref{mainsub}).

\section{The Raynaud-Mukai construction}

In this section, we present the Raynaud-Mukai construction.  Although
\cite{Mu05} is available now, we prefer to use the version described
in \cite{Mu79}, which is slightly different from the 2005 version. As
1979 version is only available in Japanese, we present some
details for the readers convenience.

The idea is to construct from a counter-example 
to Kodaira vanishing, i.e., a polarized variety $(X, L)$
with $H^1(X, L^{-1})\ne{0}$ a new counter-example 
$(\tilde{X}, \tilde{L})$ with $\dim\tilde{X}=\dim{X}+1$.
This inductive construction starts from 
a polarized curve $(X, L)$ called a Tango-Raynaud curve.

\subsection{pre-Tango structure and Kodaira non-vanishing}

\begin{Definition}[pre-Tango structure]
\label{def:pretango}
Let $X$ be a smooth projective variety. Then an ample divisor $D$,
or an ample line bundle $L = \OO_X(D)$, is called a {\em pre-Tango
structure} if there exists 
an element $\eta\in k(X)\backslash k(X)^p$, where $k(X)$ denotes the function
field of $X$, such that the K\"ahler differential
is $d\eta \in \Omega_X(-pD)$, which will be simply denoted as $(d\eta)\geq pD$.
In this paper, the element $\eta$ will be called a {\em justification} 
of the pre-Tango structure.
\end{Definition}

Existence of a pre-Tango structure implies Kodaira non-vanishing.
In fact, consider the absolute Frobenius morphism
\begin{equation*}
   F :  \OO_X(-D) \To \OO_X(-pD)
\end{equation*}
such that $F(a) = a^p$ for $a \in \OO_X$ and set $B_X(-D) := \Coker{F}$.
Then we have 
\begin{equation*}
0\To H^0(X, B_X(-D)) \To H^1(X, \OO_X(-D))\overset{F}{\To} H^1(X, \OO_X(-pD))
\end{equation*}
and then we can show 
\begin{Proposition}$H^0(X, B_X(-D)) 
= \left\{ df \in k(X)\;\vert\; (df)\geq pD\right\}$.
\end{Proposition}
Thus, if there exists a pre-Tango structure $D$ and $\dim{X}\geq 2$,
then we have Kodaira non-vanishing: $H^1(X, \OO_X(-D))\ne{0}$.

Notice that the inclusion $H^0(X, B_X(-D)) \subset H^1(X, \OO_X(-D))$
may be strict, so that there is a possibility that a non pre-Tango
structure $L$ causes a Kodaira non-vanishing. However, since 
the iterated Frobenius map
\begin{equation*}
   F^{e} : H^1(X, L^{-1}) \To H^1(X, L^{-p^e})
\end{equation*}
is trivial for $e\gg 0$, $L^{n}$ is a pre-Tango structure for 
sufficiently large $n\in\NN$.

Pre-Tango structure for curves are characterized by 
the Tango-invariant \cite{Tango72, Tango72k}. 
Let $C$ be a smooth projective curve of genus $g\; (\geq 2)$. Then 
the Tango-invariant is defined as 
\begin{equation*}
 n(C) = \max\left\{
           \deg\left[\frac{df}{p}
           \right]
           \;:\;
             f\in k(X)/k(X)^p
         \right\}
\end{equation*}
where $[\cdots]$ denotes the round up.
We easily know that
\begin{equation*}
     0\leq n(C) \leq \frac{2(g-1)}{p}.
\end{equation*}
Then, $C$ has a pre-Tango structure $D$ if $n(C)>0$.
We just set $D = \left[\frac{(df)}{p}\right]$
and then $D$ is ample on $C$
such that $(df) \geq pD$.

In the following, we will call the pair $(X, L)$ in
Definition~\ref{def:pretango} a {\em pre-Tango polarization}.
The Raynaud-Mukai construction is an algorithm to make a new pre-Tango
polarization from a pre-Tango polarization whose dimension is lower by
one.

\subsection{purely inseparable cover}

From a pre-Tango polarized variety $(X, L)$ we can construct 
a reduced and irreducible purely inseparable cover $\tau : G \To X$ of degree $p$. 
Conversely, existence of such a cover implies existence of a pre-Tango polarization.

\subsubsection{Construction and characterization} 

Given a pre-Tango polarized variety $(X, L=\OO_X(D))$,
choose an element 
$(0\ne)\eta \in H^0(X, B_X(-D)) 
(= \Ker F)$.
Then we have 
a corresponding non-split short exact sequence
\begin{equation}
\label{extension}
  0 \To \OO_X \To E \To L \To 0
\end{equation}
where $E$ is a rank~$2$ vector bundle on $X$. 
Taking the Frobenius pull-back, we obtain an exact sequence
\begin{equation*}
  0 \To \OO_X \To E^{(p)} \To L^{(p)} \To 0.
\end{equation*}
where, for example, $E^{(p)} = E\tensor_{\OO_X}\OO_{X'}$ 
with $F: \OO_{X}\To \OO_{X'}$ the Frobenius morphism.
Notice that the new sequence corresponds to $F(\eta) =0$
so that it splits and by using the split maps, we obtain the 
sequence with the reverse arrows
\begin{equation*}
  0 \oT \OO_X \oT E^{(p)} \oT L^{(p)} \oT 0.
\end{equation*}
Tensoring by ${L^{(p)}}^{-1}$ over $\OO_X$, we finally obtain
the sequence 
\begin{equation}
\label{f-extension}
  0 \To \OO_X \To E^{(p)}\tensor {L^{(p)}}^{-1} \To {L^{(p)}}^{-1} \To 0.
\end{equation}
Now we consider the $\PP^1$-fibration
\begin{equation*}
\pi: P = \PP(E)\To X
\end{equation*}
together with the canonical section $F\subset P$, which is 
defined by the image of $1\in \OO_X$ in $E$, and 
\begin{equation*}
\pi^{(p)}: P^{(p)}=\PP(E^{(p)}\tensor {L^{(p)}}^{-1})
\iso \PP(E^{(p)}) \To X
\end{equation*}
together with the canonical section $F^{(p)} \subset P^{(p)}$ 
which is the image of $1\in\OO_X$ in $E^{(p)}\tensor {L^{(p)}}^{-1}$,
corresponding to (\ref{extension}) and (\ref{f-extension}).
Moreover, we consider the relative Frobenius morphism 
$\psi : P \To P^{(p)}$ over $X$. On an open set $U\subset X$
such that $E\vert_U \iso \OO_U^{r}$ with $r =\rank{E}$, 
$\psi$ is induced by the local morphism 
$E^{(p)}\vert_U\iso \OO_U^r\tensor_{\OO_U}\OO_{U'}\to E\vert_U$ sending
$\sum_{i=1}^r a_i\tensor f = \sum_i 1\tensor a_i^pf\in \OO^r_{U}$ to $\sum_i a_i^pf
\in \OO^r_{U}$. Thus, on a fiber $\pi^{-1}(x) \iso \PP^1$,
$\psi\vert_{\pi^{-1}(x)} : \pi^{-1}(x) \To {\pi^{(p)}}^{-1}(x)$
is the Frobenius pull-back, i.e., $\psi (a,b) = (a^p, b^p)$
for every projective coordinate $(a,b) \in \pi^{-1}(x)$.
Now consider the scheme theoretic inverse image of $F^{(p)}$
inside $P$:
\begin{equation*}
      G := \psi^{-1}(F^{(p)}) \subset P
\end{equation*}
Then we can show 
\begin{Proposition} \label{picprop}
\begin{enumerate}
\item $G \cap F = \emptyset$,
\item  $\OO_P(G) \iso \OO_P(p)\tensor \pi^*{L^{-p}} \iso \OO_P( p F - p\pi^{*}D)$,
and 
\item $\rho = \pi\vert_{G} : G \To X$ is a purely inseparable cover 
of degree $p$.
\end{enumerate}
\end{Proposition}

We can show that existence of such a $G$ characterizes pre-Tango structure.
To summarize, we have 

\begin{Theorem}[See Proposition~1.1 in \cite{Mu05}]
\label{Mu05:prop1.1}
Let $X$ be a smooth projective variety of characteristic $p>0$ and 
$L$ be an ample line bundle. Then the following are equivalent:
\begin{enumerate}
\item $L$ is a pre-Tango structure.
\item There exists a $\PP^1$-bundle $\pi : P \To X$ and a reduced 
irreducible effective divisor $G \subset P$ such that 
\begin{enumerate}
\item $\rho : G \To X$ is a purely inseparable cover of degree $p$
\item $P = \PP(E)$ where $E$ is a rank~2 vector bundle on $X$ such that 
\begin{equation*}
    0 \To \OO_X \To E \To L \To 0
\end{equation*}
\end{enumerate}
\end{enumerate}
\end{Theorem}

\subsubsection{smoothness} 

For smoothness of the purely inseparable cover $G$, we have 

\begin{Theorem}[S.\ Mukai~\cite{Mu05}]
\label{smoothnessofG}
Let $(X, D)$ be a pre-Tango polarized variety over the field of 
characteristic $p>0$ and $G$ is the purely inseparable cover 
constructed from  a justification $(0\ne)\eta\in k(X)\backslash k(X)^p$.  
Then $G$ is smooth if and only if
$(d\eta) = p D$.
This means that for the multiplication by $d\eta$
\begin{equation*}
   \OO_X(pD)\overset{d\eta}{\To} \Omega_X \To \Coker(d\eta),
\end{equation*}
$\Coker{d\eta}$ is locally free at every $x\in X$.
\end{Theorem}
\begin{proof}
For a proof in the case of $\dim X=2$, see Theorem~3 \cite{T10}.
\end{proof}

\begin{Definition}[Tango structure]
Let $X$ be a smooth projective variety with a pre-Tango structure
$L = \OO_X(D)$. Then, $D$, or $L$,  is called a {\em Tango structure}
if  and only if  a justification $\eta\in k(X)\backslash k(X)^p$
satisfies  $(d\eta) = p D$. In this case, the pre-Tango polarization
$(X, L)$ or $(X,D)$ will be called a {\em Tango polarization}.
\end{Definition}

A smooth projective curve $X$ of genus $g\geq 2$ with 
a Tango structure $D$ is called a {\em Tango-Raynaud curve}.
For examples of Tango-Raynaud curves, see for example
\cite{Ray, Mu79, Mu05}.

\subsection{cyclic cover}

Let $(X, D)$ be a pre-Tango polarization and 
$D$ is divided by $k\in\NN$ 
with $(p,k)=1$ and we have $D = k D'$.
If $X$ is a curve, we can divide $D$ by any natural number $k$ 
dividing  $\deg D$  using the theory of Jacobian variety
(cf. page~62 of \cite{Ab}). But the condition $(p,k)=1$ 
is necessary for the covering to be cyclic.

Now we construct a $k$th cyclic cover of the $\PP^1$-fibration
$\pi : P\To X$ ramified over $F + G$,
which means that $\pi$  is ramified at the reduced preimage of $F + G$.
There are at least two well-known constructions.

The first one is rather explicit and is suitable 
for computing cohomologies (cf. \cite{T10}). We first 
choose $m\in\NN$ such that $k \vert (p+m)$ and set 
${\cal M} = \OO_P\left(-\frac{p+m}{k}F\right) 
               \tensor \pi^*\OO_X(pD')$.
Then we have 
${\cal M}^{\tensor{k}} = \OO_P(-mF)\tensor \OO_P(-pF)\tensor \pi^*\OO_X(pD)
        = \OO_P(-mF -G)$
by Proposition~\ref{picprop}. Then we can introduce 
    $\Dirsum_{i=0}^{k-1}{\cal M}^{\tensor{i}}$
the structure of a graded $\OO_P$-algebra by defining multiplication
     ${\cal M}^{\tensor{i}}\times {\cal M}^{\tensor{j}}
\To {\cal M}^{\tensor{i+j}}$
s. t. $(a. b) \mapsto a\tensor b$
if $i+j<k$ and 
     ${\cal M}^{\tensor{i}}\times {\cal M}^{\tensor{j}}
\To {\cal M}^{\tensor{i+j}} \To  {\cal M}^{\tensor{i+j-k}}$
s. t. $(a. b) \mapsto a\tensor b \mapsto a\tensor b \tensor \xi$
if $i+j\geq k$ where we choose a non-trivial  element $\xi \in \OO_P(mF +G)$
such that $mF + G$ is the zero locus of $\xi$.
Now we consider the affine morphism 
    $X' := {\cal Spec}\; \Dirsum_{i=0}^{k-1}{\cal M}^{\tensor{i}}\to P$
and this is the cyclic cover ramified over  $mF + G$.
Since $X$ is smooth, $F\iso X$ is also smooth.
Moreover if $D$ is a Tango structure and $G$ is smooth 
by Theorem~\ref{smoothnessofG}, then $X'$ is smooth if and only if $m=1$;
if $m>1$ then $X'$ is singular along $F$, which may cause non-normality
of $X'$. Normalization of $X'$, if necessary, is carried out by Esnault-Viehwegs 
method
(see \S~3 of \cite{EV}). 
   $\tilde{X} = {\cal Spec} \Dirsum_{i=0}^{k-1} M^{\tensor{i}}
\tensor\OO_P\left(
              \left[
                  \frac{i ( mF + G)}{k}
             \right]
          \right)$
and this is smooth if $D$ is Tango.
We note that this normalization procedure highly depends on
the condition $(p,k)=1$ since we use the $k$th root of unity.
Then we set the natural morphism 
$\varphi: \tilde{X} \To X' \To P$.

The second construction uses normalization. Since we have liner equivalence
$G \sim p F - p \pi^*(D)$
there exist a function $R \in k(P)$ such that 
$(R) =  G - (p F - p \pi^*(D))
        =  G - (p F - p k \pi^*(D'))$.
Then let $\tilde{X}$ be the normalization of $P$
in the finite extension $k(P)(R^{1/k})$ of $k(P)$ and 
$\varphi : \tilde{X} \To P$
be the normalization morphism. Then we set $f= \pi\circ \varphi$.
Now if we work locally we know that there exist divisors 
$\tilde{G}$ and $\tilde{F}$ on $\tilde{X}$ such that 
$\varphi^*F = k \tilde{F}$ and $\varphi^*G = k \tilde{G}$. 
Moreover, we have 
$\tilde{G} \sim p \tilde{F} - p f^*(D')$
on $\tilde{X}$. 
We note that the condition  $(p,k)=1$ is necessary to 
assure  the existence of $\tilde{F}$, division of $F$ by $k$.
Otherwise, if $k =p^\ell r$ with $\ell\geq 1$ and $(p,r)=1$
we have $\tilde{F}\subset \tilde{X}$ such 
that $\varphi^*F = k' \tilde{F}$ with $k' = p^{\ell-1}r = k/p$.
$\tilde{X}$ is smooth if $D$ is Tango.

Now we set  $f := \pi\circ\varphi:\tilde{X}\To X$, which is 
actually a fibration of rational curves with moving singularities,
i.e., rational curves with cusp singularity of type $x^p = y^t$
at $\tilde{G}$.

\subsection{polarization}
The cyclic cover $\tilde{X}$ of the $\PP^1$-fibration is a counter-example to 
Kodaira vanishing because the polarization 
$\tilde{D} = (k-1)\tilde{F} + f^*(D')$ causes non-vanishing 
$H^1(\tilde{X}, \OO_{\tilde{X}}(-\tilde{D}))\ne{0}$. In fact,
$\tilde{D}$ is ample (see Sublemma~1.6 \cite{Mu05})
and we have 
\begin{Proposition}
\label{newTangStr}
Suppose $\tilde{X}$ is as above, then
$\tilde{D}$ is a Tango structure of $\tilde{X}$
and in particular we have Kodaira non-vanishing 
$H^1(\tilde{X}, \OO_{\tilde{X}}(-\tilde{D}))\ne{0}$.
\end{Proposition}

This result is stated in \cite{Mu79} without proof and in the case of 
$k\equiv 1 \mod p$ a proof using Maruyama's elementary transformation
\cite{Maru} is given in \cite{Mu05}. We give here a proof of general case.

\begin{proof}
Let $\tilde\eta = R^{1/k}\in k(\tilde{X})$.
Since $(\tilde\eta)= \tilde{G} - p \tilde{F} + p f^*(D')$,
$\tilde\eta$ is locally described as 
$\tilde\eta = g(\delta\phi^{-1})^p$
where $g$, $\phi$ and $\delta$ are local equations 
defining $\tilde{G}$, $\tilde{F}$ and $f^*(D')$. Then 
its K\"ahler differential is 
\begin{equation}
\label{kaehlerdiff}
       d\tilde\eta = (\delta\phi^{-1})^p dg
                   = (\phi^{k-1}\delta)^p\phi^{-pk}dg.
\end{equation}
Now we consider $dg$. 
As a Cartier divisor we describe $D = \{(U_i, g_i)\}_i$ 
for an open cover $X = \Union_i U_i$ and $g_i\in k(X)$.
Since $D$ is a pre-Tango structure, there exists a justification
$\eta\in k(X)$ such that $(d\eta) \geq pD$, which locally 
means that we have $\eta\vert_{U_i} = g_i^pc_i$ for some $c_i \in \OO_{U_i}$
so that we have $(d\eta)\vert_{U_i} = g_i^p dc_i$.
Then, as in Proposition~1 \cite{T10}, $G\subset P$
is locally described as 
\begin{equation*}
     \Proj \OO_{U_i}[x,y]/(c_i x^p + y^p)
\end{equation*}
where $x$ is the (local) coordinate corresponding to 
the canonical section $F$ of $\pi: P\To X$.
Hence the local defining equation of $G \subset P$  is 
$c_i x^p + y^p$, 
and since $\varphi^{*}F = k \tilde{F}$
and $\varphi^{*}G = k \tilde{G}$, the defining equation 
of $\tilde{G}$ is 
$g = c_i Z^{kp} + W^{kp}$,
where $Z$ is the local coordinate of $\tilde{X}$ corresponding to
$\tilde{F}$, namely $Z = \epsilon\phi$ with some local unit $\epsilon$.
Thus we have 
\begin{equation}
\label{dg}
   dg = \epsilon^{pk}\phi^{pk} dc_i.
\end{equation}
Thus by (\ref{kaehlerdiff}) and (\ref{dg}) we obtain
$d\tilde\eta = (\delta\phi^{-1})^p dg
                   = \epsilon^{pk}(\phi^{k-1}\delta)^p dc_i$
so that
\begin{equation*}
(d\tilde\eta) \geq  p((k-1)\tilde{F} + f^*D')  = p\tilde{D}
\end{equation*}
where the equality holds if $(d\eta) =pD$, i.e., if $D$ is 
a Tango structure.
\end{proof}

\section{Calabi-Yau threefolds and the Raynaud-Mukai construction}

\subsection{Raynaud-Mukai varieties cannot be Calabi-Yau}

The aim of this section is to show that 
Mukai construction does not produce K3 surfaces or 
Calabi-Yau threefolds. Notice that Raynaud-Mukai variety is 
always of general type for $p\geq 5$ 
(cf. Prop.~7~\cite{Mu79} or Prop.~2.6~\cite{Mu05})
so that the only 
possibility is the case $p=2, 3$.

Now let $(X, D)$,  $D = kD'$ ($k\in\NN$), 
$\pi:P\To X$,  $F, G\subset P$, $(\tilde{X}, \tilde{D})$,
$\varphi: \tilde{X}\To P$
and $f: \tilde{X}\To X$ be as in the previous section.
The canonical divisor of $X$ will be simply denoted by $K$.
Now we have 

\begin{Proposition}[cf. Prop.~7~\cite{Mu79}]
\label{Mukai79:prop7}
Let $\tilde{K}$ be the canonical divisor of $\tilde{X}$. Then we have 
\begin{equation*}
\tilde{K} \sim 
(pk -p - k-1)\tilde{F}
+ f^*(K- (pk- p - k)D')
\end{equation*}
\end{Proposition}
\begin{proof}
Since the finite morphism $\varphi:\tilde{X}\To P$ is 
ramified at $\tilde{F}=(\varphi^*(F))_{red}$ and 
$\tilde{G}=(\varphi^*(G))_{red}$ with the same ramification index $k$
and $F+ G \sim (p+1)F - pk\pi^*D'$, we compute
\begin{eqnarray*}
\tilde{K}
&\sim & \varphi^*K_P + (k-1)(\tilde{F} + \tilde{G}) \quad\mbox{by ramification formula}\\
&\sim & \varphi^*K_P + (k-1)\frac{1}{k}\varphi^*(F+G)\\
&\sim & \varphi^*K_p + (k-1)((p+1)\tilde{F} - p f^*D').
\end{eqnarray*}
Moreover, since $E$ is the rank~$2$ vector bundle satisfying 
\begin{equation*}
    0 \To \OO_{X} \To E \To \OO_{X}(kD')\To 0,
\end{equation*}
we have $K_{P}\sim -2 F+ \pi^*(K + kD')$.
Then we obtain the required formula.
\end{proof}

We notice that since $\Pic P \iso \ZZ\cdot[F]\dirsum \pi^*\Pic{X}$ and 
$\varphi$ is finite, we have $\Pic{\tilde{X}} \iso \ZZ\cdot [\tilde{F}]
\dirsum f^{*}\Pic{X}$. This fact will be used implicitly in the following 
discussion.

\begin{Corollary}
\label{neverK3}
A Raynaud-Mukai surface can never be a K3 surface.
\end{Corollary}
\begin{proof}
Assuming $\dim\tilde{X}=2$, 
we have only to show that we never have $\tilde{K}\sim 0$.
Assume that we have 
$\tilde{K} \sim 
(pk -p - k-1)\tilde{F} + f^*(K - (pk - p - k)D') \sim 0$,
from which have two relations $pk -p - k-1 =0$ and $K - (pk - p - k)D'=0$.
By the first relation, we have 
$k = \frac{p+1}{p-1} \in \NN$,
so that we must have  $p=2$ and $k= 3$ or $p=3$ and $k = 2$.
This implies $K = D'$ by the second relation.
However,  since $(X, D)$ is a (pre-)Tango polarized curve, we have 
$(d\eta) \geq pD$ for some justification $\eta\in k(X)$, namely 
$D' = K \geq  p D = pk D'$, which is impossible unless $pk=1$.
\end{proof}

By a similar discussion to the proof of Corollary~\ref{neverK3}, 
we can also show 

\begin{Corollary}
\label{trivialK}
A Raynaud-Mukai threefold can never be Calabi-Yau.
\end{Corollary}
\begin{proof}
Let $\tilde{X}$ be a Mukai threefold obtained from a Mukai
surface $X$ with a (pre-)Tango structure $D = kD'$
as a $k$th cyclic cover of the $\PP^1$-fibration $P$
and assume that $\tilde{K}\sim 0$. Then as in the proof of 
Corollary~\ref{neverK3} we have 
$(p,k) = (2,3)$ or $(3,2)$ and  
\begin{equation}
\label{ksimddash}
K\sim D'.
\end{equation}
Now we will consider the situation whose dimensions are all lower by one.
Namely, let the surface $X$ be constructed from a (pre-)Tango polarized curve 
$(X_1, D_1)$ with $D_1 = k_1D_1'$. We have the $k_1$th cyclic cover 
$\varphi_1: X \To P_1$ of the $\PP^1$-fibering 
$\pi_1 : P_1\To X_1$ ramified over $F_1 + G_1$ and 
$\tilde{F}_1 = (\varphi_1^*(F_1))_{red}$ and $\tilde{G}_1 = (\varphi_1^*(G_1))_{red}$
have the same ramified index $k_1$. We set $f_1 = \pi_1\circ\varphi_1$.
Then by Proposition~\ref{Mukai79:prop7}, we have 
\begin{equation*}
K \sim (pk_1 - p - k_1 -1)\tilde{F}_1 + f_1^*(K_1 - (pk_1 - p - k_1)D_1')
\end{equation*}
Since  we have $(kD'=) D = (k_1-1)\tilde{F}_1 + f_1^*(D_1')$ by definition,
the condition (\ref{ksimddash}) entails
\begin{equation*}
\left(pk_1-p-k_1-1 -  \frac{k_1-1}{k}\right)\tilde{F}_1
+ f_1^*
\left(
K_1 - (pk_1 - p - k_1 + \frac{1}{k})D_1' 
\right)
\sim 0.
\end{equation*}
Then the coefficient of $\tilde{F}_1$ must be $0$ so that we have 
\begin{equation*}
 k_1 =  \frac{k(p+1)-1}{k(p-1)-1}   
     = \left\{
             \begin{array}{ll}
                4 & \mbox{if $p=2$}\\ 
                  & \\
               \displaystyle{\frac{7}{3}} & \mbox{if $p=3$}
             \end{array}
          \right.
\end{equation*}
But since we must have $k_1\in \NN$ and $(k_1, p)=1$, these values of $k_1$ 
are not allowed.
\end{proof}

\subsection{a modification of the Raynaud-Mukai construction}

The Raynaud-Mukai construction is an
algorithm to construct 
from a given (pre-)Tango polarization $(X, D)$ with $D = kD'$
a new  (pre-)Tango polarization
$(\tilde{X},\tilde{D})$ with $\dim X = \dim \tilde{X}-1$ 
by taking a $k$th cyclic cover. We apply this procedure inductively 
starting from a (pre-)Tango polarized curve.
We have seen in the previous subsection that the essential reason that
the Raynaud-Mukai construction does not produce Calabi-Yau threefolds is
that we cannot find the degree $k$  cyclic covers with $(p,k)=1$ 
in all inductive steps.

Now we will consider  some modification of the Raynaud-Mukai construction.
There are following two possibilities.

\begin{enumerate}
\item [(I)] Let $(X, D)$ be a (pre-)Tango polarized surface 
obtained by a method other than Mukai construction. Then apply 
the Raynaud-Mukai construction to obtain a (pre-)Tango
polarized threefold $(\tilde{X}, \tilde{D})$.

\item [(II)] Let $(X, D)$ be a (pre-)Tango  polarized  surface 
by the Raynaud-Mukai construction. Then we construct a Calabi-Yau threefold
in a similar way to Mukai construction. Namely, we do not assume 
the condition $(p,k)=1$ for the degree $k$ of ``cyclic cover ``.
\end{enumerate}

The Calabi-Yau threefolds obtained  by $(I)$ are counter-examples 
to Kodaira vanishing.  The surface $X$ required in $(I)$ is precisely 
as follows:
\begin{Corollary}
\label{requiredsurface}
Let $(X, D)$ a (pre-)Tango polarized surface with  $D = kD'$ for some $k\in \NN$.
Then  the Raynaud-Mukai construction gives a polarized Calabi-Yau  
threefold $(\tilde{X}, \tilde{D})$ by a $k$th cyclic cover if and only if 
\begin{enumerate}
\item [$(i)$] $(p,k)= (2,3)$ or $(3,2)$, and 
\item [$(ii)$] $D = k D'$  for some ample $D'$ and $K_X \sim D'$.
\end{enumerate}
In particular, $X$ is a surface of general type.
\end{Corollary}
\begin{proof}
By the same discussion as in 
the proof of Corollary~\ref{neverK3} and ~\ref{trivialK}.
\end{proof}

Unfortunately we do not know how to construct a polarized surface $(X,
D)$ as in Corollary~\ref{requiredsurface}. 
But Theorem~\ref{NonKVsurfaces}$(i)$ below seems to indicate a possibility.

\begin{Theorem}[S.\ Mukai \cite{Mu79}]
\label{NonKVsurfaces}
Let $X$ be a (smooth) surface over the field $k$ of $\chara{k}=p>0$. Assume that 
Kodaira vanishing fails on $X$.
Then we have
\begin{enumerate}
\item [$(i)$] $X$ is of general type or quasi-elliptic surface 
with Kodaira dimension $1$ (if $p=2,3$).
\item [$(ii)$] There exists a surface
 $X'$ birationally equivalent with $X$
such that there is a morphism $g : X'\To C$ to a curve $C$
whose fibers are all connected and singular.
\end{enumerate}
\end{Theorem}

It is proved that, in the case of surfaces, Kodaira (non-)vanishing
is preserved in birational equivalence
(see Corollary~8 \cite{Tango72k}). Thus by Theorem~\ref{NonKVsurfaces}$(ii)$
it seems to be reasonable to consider a fibration $\rho : X \To C$ to a curve.

For a Calabi-Yau threefold, we often assume simple connectedness
which implies $H^1(\tilde{X}, \OO_{\tilde{X}})=0$ for our example.
For this property, we have the following. 
\begin{Proposition}
\label{h1ofCY3fold}
Assume that the surface $X$ in Corollary~\ref{requiredsurface}
has a fibration over a curve $C$:  $g: X\To C$ and 
set $h : \tilde{X}\overset{f}{\To}X\overset{g}{\To}C$.
Then we have 
$H^1(\tilde{X}, \OO_{\tilde{X}})
\iso H^1(C, g_*\OO_X)\dirsum 
H^0(C, R^1 h_*\OO_{\tilde{X}})$.
\end{Proposition}
\begin{proof}
Consider the Leray spectral sequence
\begin{equation*}
    E_2^{pq} = H^p(C, R^qh_*\OO_{\tilde{X}}) 
\Rightarrow H^{p+q}(\tilde{X}, \OO_{\tilde{X}}).
\end{equation*}
Then by the 5-term exact sequence we have 
\begin{equation*}
0\To  H^1(C, h_*\OO_{\tilde{X}}) \To H^1(\tilde{X}, \OO_{\tilde{X}}).
 \To H^0(C, R^1h_*\OO_{\tilde{X}})
 \To H^2(C, h_*\OO_{\tilde{X}})
\end{equation*}
where the last term $H^2(C, h_*\OO_{\tilde{X}})$ vanishes since $\dim C < 2$.
Thus we have 
\begin{equation*}
H^1(\tilde{X}, \OO_{\tilde{X}})
\iso H^1(C, h_*\OO_{\tilde{X}})\dirsum 
H^0(C, R^1 h_*\OO_{\tilde{X}}).
\end{equation*}
On the other hand,  we have $(p,k)=(2,3)$ or $(3,2)$ by
Corollary~\ref{requiredsurface} and 
the explicit construction of the cyclic 
cover gives
\begin{eqnarray*}
\tilde{X}
&=& 
\left\{
\begin{array}{ll}
{\cal Spec}\Dirsum_{i=0}^2\OO_P(-i)\tensor\pi^*(2iD')
 & \mbox{if $(p,k)= (2,3)$}
\\
{\cal Spec}\Dirsum_{i=0}^1\OO_P(-2i)\tensor\pi^*(3iD')
 & \mbox{if $(p,k)= (3,2)$}
\end{array}
\right.
\end{eqnarray*}
where $\pi: P \To X$ is the $\PP^1$-fibering.
Thus we compute
\begin{eqnarray*}
h_*\OO_{\tilde{X}} 
&=& (g\circ \pi\circ \varphi)_*\OO_{\tilde{X}}
 =  (g\circ \pi)_* (\varphi_*\OO_{\tilde{X}})\\
&=& \left\{
\begin{array}{ll}
(g\circ\pi)_*\left(\Dirsum_{i=0}^2\OO_P(-i)\tensor\pi^*(2iD')\right)
 & \mbox{if $(p,k)= (2,3)$}
\\
(g\circ\pi)_*\left(\Dirsum_{i=0}^1\OO_P(-2i)\tensor\pi^*(3iD')\right)
 & \mbox{if $(p,k)= (3,2)$}
\end{array}
\right.\\
&=& \left\{
\begin{array}{ll}
g_*(\pi_*\OO_P) \dirsum  \; 
g_* (\pi_*\OO_P(-1)\tensor\OO_X(2D')) & \\
\hspace*{20pt} \dirsum \;
g_* (\pi_*\OO_P(-2)\tensor\OO_X(4D'))
 & \mbox{if $(p,k)= (2,3)$}
\\
g_* (\pi_*\OO_P)
\dirsum
g_* (\pi_*\OO_P(-2)\tensor\OO_X(3D'))
 & \mbox{if $(p,k)= (3,2)$}
\end{array}
\right.
\\
\end{eqnarray*}
Now since $\pi_*\OO_P = \OO_X$ and $\pi_*\OO_P(-i)=0$ for $i>0$
we obtain $h_*\OO_{\tilde{X}}= g_*\OO_X$.
\end{proof}

\begin{Remark}{\em 
Using another spectral sequence and 5-term exact sequence 
we can show the inclusion $H^0(C, R^1g_*(f_*\OO_{\tilde{X}}))
\subset H^0(C, R^1h_*\OO_{\tilde{X}})$ but the equality does not 
hold in general.
}\end{Remark}

Next we consider the construction $(II)$, whose algorithm is as follows:
Given a (pre-)Tango curve, we make a (pre-)Tango polarized surface $(X, D)$
and a $\PP^1$-bundle $\pi: P\To X$ with the canonical section $F\subset P$ 
together with a purely inseparable
cover $\pi\vert_G : G\To X$  of degree $p$ corresponding to $D$.
Then choose $k=p^\ell r$ with $(p,r)=1$ and $\ell\geq 1$ and 
let $\varphi : \tilde{X}\To P$ be the normalization of $P$ in 
$k(P)(R^{1/k})$ where $R\in K(P)$ is such that 
$(R) = G - (pF -p\pi^*(D))$.

\begin{Lemma}
\label{algorithmII}
Let $(X, D)$, $D =kD'$ with $(2\leq)\; k\in\NN$, 
be a (pre-)Tango polarized surface by
the Raynaud-Mukai construction.
Then the construction $(II)$ gives a Calabi-Yau
threefold if and only if 
 $(p,k, K)=(2,4, 2D')$ or $(3,3, D)$.
\end{Lemma}
\begin{proof}
Let $(X, D)$ be a
(pre-)Tango polarized surface by the Raynaud-Mukai construction.
Then we obtain  a $\PP^1$-bundle
$\pi : P\To X$ together with the canonical section $F$ 
and the purely inseparable cover $G \to X$ of degree $p$
(see Theorem~\ref{Mu05:prop1.1}).

In Mukai construction, we take a $k$th cyclic cover 
of $P$ where $(k,p)=1$. This does not work as we have seen 
in Corollary~\ref{trivialK}. Thus we assume  $(k,p)\ne{1}$ and set 
$k = p^\ell r$ with $(p,r)=1$, $\ell\geq 1$.
Since we have  $D = kD'$ and $G \sim p F - p \pi^*(D)$,
there exists $R\in k(P)$ such that 
$(R) = G - p F + p \pi^*(kD')$.
Now let $\varphi: \tilde{X}\To P$  be the normalization of 
$P$ in $k(P)(R^{1/k})$.  Then if we set $\tilde{F} =(\varphi^*(F))_{red}$
 and $\tilde{G}=(\varphi^*(G))_{red}$, we have 
$\varphi^*(G) = k \tilde{G}$ and $\varphi^*(F) = (k/p)\tilde{F}$
and $\tilde{G} \sim \tilde{F} - p f^*(D')$
where $f = \pi\circ\varphi$.
Notice that we do not have the coefficient $p$ for $\tilde{F}$
as in the case of $(p, k)=1$.
Now as in proof of Proposition~\ref{Mukai79:prop7}, we compute
\begin{eqnarray*}
\tilde{K}
&\sim & \varphi^*K_P + (k-1)\tilde{G} + (\frac{k}{p}- 1)\tilde{F} \\
&\sim & \varphi^*K_P + 
         \left(k + \frac{k}{p} -2\right) \tilde{F}
        - p(k-1)  f^*D'\\
&\sim& (p^{\ell}r - p^{\ell-1}r -2)\tilde{F}
   + f^*(K + (p^{\ell}r - p (p^{\ell}r -1))D').
\end{eqnarray*}
Then if $\tilde{X}$ is a Calabi-Yau threefold, i.e.,
$\tilde{K}\sim 0$, we must have 
$p^{\ell}r - p^{\ell-1}r -2=0$ 
and $K + (p^{\ell}r - p (p^{\ell}r -1))D'\sim 0$,
from which we have 
$(\ell, r, p,k)= (1, 1, 3, 3)$ or $(2, 1, 2,4)$ and 
\begin{equation*}
K\sim 
\left\{
\begin{array}{ll}
    2 D'   & \mbox{if }(p,k) = (2,4)\\
    3 D'(=D)   & \mbox{if }(p,k) = (3,3)
\end{array}
\right.
\end{equation*}
\end{proof}

Now we can show 
\begin{Proposition}
\label{cy3impossible}
Calabi-Yau threefolds cannot be obtained by the construction (II).
\end{Proposition}
\begin{proof}
We assume that the (pre-)Tango polarized surface $(X, D)$ 
is a fibration $f_1 : X\To X_1$ over a Tango polarized curve 
$(X_1, D_1)$ with $D_1 = k_1D_1'$, which is a $k_1$th cyclic
cover $\varphi_1: X\To P_1$ of a $\PP^1$-fibration $\pi_1:P_1 \To X_1$ 
ramified over $F_1 + G_1\subset P_1$ and we set $\tilde{F}_1 = (\varphi_1^*(F_1))_{red}$.
In this situation, we have 
\begin{equation*}
K \sim (pk_1 -p - k_1-1)\tilde{F}_1+ f_1^*(K_{X_1}- (pk_1- p - k_1)D_1')
\end{equation*}
by Proposition~\ref{Mukai79:prop7}.
We have $D = (k_1-1)\tilde{F}_1 + f_1^*D_1'$ by definition.
Now we first consider the case $(p,k) = (2,4)$. By 
Lemma~\ref{algorithmII} we have 
\begin{equation*}
2D' = \frac{1}{2}D = \frac{1}{2}(k_1-1)\tilde{F}_1 + \frac{1}{2}f_1^*D_1' 
\sim K= (k_1-3)\tilde{F}_1+ f_1^*(K_{X_1}- (k_1-2) D_1')
\end{equation*}
or otherwise
\begin{equation*}
\frac{5-k_1}{2}\tilde{F}_1 + f_1^*(\frac{2k_1-3}{2}D_1' - K_{X_1}) \sim 0,
\end{equation*}
which entails $k_1=5$ and $K_{X_1}= \frac{7}{2}D_1'$.
But since $D_1$ is a (pre-)Tango structure we must have 
$\frac{7}{2}D_1' = K_{X_1}\geq pD_1 = 2\cdot 4 D_1' = 8D_1'$, a contradiction.

The case of $(p,k)=(3,3)$ is similar. 
Since we must have $D \sim K$, we have 
$k_1 = 3$ and $K_{X_1} =4D'$.
But, since $(X_1, D_1)$ is a Tango-Raynaud curve, we must have 
$4D_1' = K_{X_1} \geq p D_1 = 3k_1 D_1' = 9D_1'$, a contradiction.
\end{proof}

\section{Cohomology of Calabi-Yau threefold with Tango-structure
\label{section:quasiKV}}

In this section, we  compute the cohomology
$H^1(X, H^{-1})$ for arbitrary ample $H$ under the assumption
that $X$ is a Calabi-Yau threefold on which Kodaira vanishing 
fails.

\begin{Theorem}[N.~Shepherd-Barron \cite{SB}]
\label{SBvanishing}
Let $X$ be a normal locally complete intersection 
Fano threefold over the field $k$ of $\chara{k}=p\geq 5$
and $L$ be an ample line bundle on $X$.
Then we have $H^1(X, L^{-1})=0$.
\end{Theorem}

Recall that, for a polarized smooth variety $(X, L)$, 
Kodaira non-vanishing $H^1(X, L^{-1})\ne{0}$ does not necessarily
imply $L$ is a (pre-)Tango structure. 
But by Enriques-Severi-Zariski's  theorem, there exists $\ell>0$ 
such that  we have $H^1(X, L^{-p^{\ell+1}})=0$ but 
$H^1(X, L^{-p^\ell})\ne{0}$. Then such $L^{\ell}$ is at least 
a pre-Tango structure. Now based on these observations,
we obtain

\begin{Theorem}
\label{main}
Let $(X,L)$ be a smooth Calabi-Yau threefold  over a field $k$ 
of $\chara{k}=p\geq 5$ with Kodaira non-vanishing $H^1(X, L^{-1})\ne{0}$.
If $L^{\ell}$ is a Tango structure for some $\ell \geq 1$, then 
we have 
\begin{equation*}
H^1(X, H^{-1}) = H^0(X, H^{-1}\tensor(\rho_*\OO_Y/\OO_X))
\end{equation*}
for every ample line bundle
$H$ on $X$, where $\rho: Y\To X$ is a purely inseparable cover of degree $p$ 
corresponding to the Tango structure as in Theorem~\ref{Mu05:prop1.1}.
\end{Theorem}

\begin{proof} By taking a sufficiently large power $L^{\ell}$, $\ell\gg 0$, 
we can assume from the beginning that $H^1(X, L^{-p})=0$.
Also, by the assumption we can assume that $L$ is a Tango structure.
Then by Theorem~\ref{Mu05:prop1.1}
we have a purely inseparable cover $\rho:Y\To X$ of degree $p$ and 
$\omega_Y \iso \rho^*(\omega_X \tensor L^{-p+1})
\iso (\rho^*L)^{-p+1}$, see II~6.1.6~\cite{Ko}.
Since $\rho$ is a finite morphism and $L$ is ample, $\rho^*L$ is also ample.
Thus we know that $Y$ is an integral Fano threefold. Also since $L$ 
is a Tango structure, $Y$ is smooth by Theorem~\ref{smoothnessofG}.
Now let $H$ be an arbitrary ample line bundle on $X$. Then,
since $\rho$ is surjective, we have the following exact sequence
\begin{equation*}
  0\To H^{-1} \To H^{-1}\tensor \rho_*\OO_Y \To H^{-1}\tensor\rho_*\OO_Y/\OO_X\To 0,
\end{equation*}
from which we obtain the long exact sequence
\begin{equation*}
H^0(X, \rho_*\rho^*H^{-1})\To H^0(X, H^{-1}\tensor\rho_*\OO_Y/\OO_X)
\To H^1(X, H^{-1}) \To H^1(X, \rho_*\rho^*H^{-1}).
\end{equation*}
Now, we have $H^0(X, \rho_*\rho^*H^{-1}) = H^0(Y, \rho^*H^{-1})=0$
since $\rho$ is finite and $H$ is ample.
Also $H^1(X, \rho_*\rho^*H^{-1}) = H^1(Y, \rho^*H^{-1})$
and this is $0$ by Theorem~\ref{SBvanishing}.
\end{proof}

Recall that  for a purely inseparable cover $p : Y\To X$ of degree $p$ there exists 
a $p$-closed rational vector field $D$ on $X$ such that 
$(\rho_*\OO_Y)^D := \{ f\in \rho_*\OO_Y\;:\; D(f)=0\} = \OO_X$ (cf. \cite{russ}). Thus we have 

\begin{Corollary}
\label{mainsub}
Under the same assumption as 
Theorem~\ref{main},  we have 
\begin{equation*}
H^1(X, H^{-1}) = H^0(X, H^{-1}\tensor{D(\rho_*\OO_Y)})
\end{equation*}
where $D$ is a $p$-closed rational vector field on $X$
corresponding to the purely inseparable cover $\rho$.
\end{Corollary}

\section*{Acknowledgements}
The author thanks the referee for many careful comments. 
He also thanks  Prof.~Dr.~Holger~Brenner for stimulating discussions.


\end{document}